\documentclass{article} 
\usepackage{preamble}
\usepackage[backend=biber,style=alphabetic]{biblatex}

\bibliography{main}

\begin{document}

\title{Bigraded Poincar\'e polynomials and the equivariant cohomology of Rep($C_2$)-complexes}
\author{Eric Hogle}

\abstract{We are interested in computing the Bredon cohomology with coefficients in the constant Mackey functor $\underline{ \mathbb{F}_2}$ for equivariant $\text{Rep}(C_2)$ spaces, in particular for Grassmannian manifolds of the form $\Gr_k(V)$ where $V$ is some real representation of $C_2$.\\

It is possible to create multiple distinct $\text{Rep}(C_2)$ constructions of (and hence multiple filtration spectral sequences for) a given Grassmannian. For sufficiently small examples one may exhaustively compute all possible outcomes of each spectral sequence and determine if there exists a unique common answer. However, the complexity of such a computation quickly balloons in time and memory requirements.\\

We introduce a statistic on $\Mt$-modules valued in the polynomial ring $\Z[x,y]$ which makes cohomology computation of Rep($C_2$)-complexes more tractable, and we present some new results for Grassmannians.}

\maketitle

\tableofcontents

%%==================================%%
%% INTRODUCTION %%
%%==================================%%

\section{Introduction} 

The theory of $C_2$-equivariant spaces and spectra has been a growing area of research especially since the resolution of the Kervaire invariant one problem in \cite{HHR}. Computations of the Bredon cohomology of these objects using the constant Mackey functor $\underline{\mathbb{F}_2}$ are nontrivial, but inroads have been made, as by \cite{dL}, \cite{CH1} and  \cite{Clover}. This theory continues to prove fruitful, as in \cite{BW}, \cite{dFO} and \cite{SP}.

In particular, work on equivariant vector bundles and Grassmannians has been done as by \cite{Dan}, \cite{NG}, \cite{CH2} and \cite{K} but much work remains. While the cohomology of infinite-dimensional Grassmannians is solved in \cite{Dan} and the author of this paper worked out a two families of finite-dimensional spaces in \cite{Me}, in general the Bredon cohomology of finite-dimensional Grassmannians $\Gr_k(\R^{p,q})$ is unknown. 

Equivariant filtration spectral sequences can be built for these spaces  using a Rep($C_2$)-complex construction (see Example \ref{ex:schubert} in this paper and Example 4.1 in \cite{Dan}), but actually running the sequences is quite difficult. Prior to the results below, given an $E_1$ page and a candidate answer, even determining whether such an answer was \emph{possible} required an exhaustive investigation of all combinations of differentials that could be supported and the ``Kronholm shifts.'' of these hypothetical differentials. This brute-force approach is slow. 

We present here an invariant which makes answering the question ``could this $E_1$ converge to that module'' fast and easy. With this tool in hand a number of new cohomologies are calculated. 

%%==================================%%
%% PRELIMINARIES %%
%%==================================%%

\section{Preliminaries}

The group $C_2$ has two real representations, the trivial representation $\R^+$ and the sign representation $\R^-$. The $RO(C_2)$-graded Bredon cohomology for $C_2$-spaces is therefore bigraded. We used the motivic grading convention in which $\R^{p,q}$ denotes $(\R^+)^{\oplus(p-q)}\oplus (\R^-)^{\oplus q}$, so that the two numbers indicate underlying topological dimension and the number of dimensions with nontrivial action. We may also wish to specify a chosen order to a direct sum, for example writing $\R^{-+-}$ to mean $\R^-\oplus\R^+\oplus\R^-$. The one-point compactification of these representations are the spheres $S^{p,q}=\widehat{\R^{p,q}}$.

In this paper we abbreviate singular cohomology with $\mathbb{F}_2$ coefficients of a space $X$ by $H^\ast(X)$, and for an equivariant space $X$, write $H^{\ast,\ast}(X):=H^{\ast,\ast}(X;\underline{ \mathbb{F}_2})$. The cohomologies of equivariant spaces in this theory are modules over the cohomology of a point, shown in Figure \ref{M2}.

\begin{figure}[h]
	\begin{center}
	\begin{tikzpicture}[scale=0.4]
	\draw[gray] (-3.5,0) -- (4.5,0) node[right,black] {\small $p$};
	\draw[gray] (0,-4.5) -- (0,4.5) node[above,black] {\small $q$};
	\foreach \x in {-3,...,-1,1,2,...,4}
		\draw [font=\tiny, gray] (\x cm,2pt) -- (\x cm,-2pt) node[anchor=north] {$\x$};
	\foreach \y in {-4,-3,-2,-1,2,3,4}
		\draw [font=\tiny, gray] (2pt,\y cm) -- (-2pt,\y cm) node[anchor=east] {$\y$};
	
	\foreach \y in {0,...,4}
		\fill (0,\y) circle(2pt);
	\foreach \y in {1,...,4}
		\fill (1,\y) circle(2pt);
	\foreach \y in {2,...,4}
		\fill (2,\y) circle(2pt);
	\foreach \y in {3,...,4}
		\fill (3,\y) circle(2pt);
	\foreach \y in {4,...,4}
		\fill (4,\y) circle(2pt);
	
	\foreach \y in {0,...,2}
		\fill (0,-\y-2) circle(2pt);
	\foreach \y in {1,...,2}
		\fill (-1,-\y-2) circle(2pt);
	\foreach \y in {2,...,2}
		\fill (-2,-\y-2) circle(2pt);
	
	\draw[thick,->] (0,0) -- (4.5,4.5);
	\draw[thick,->] (0,0) -- (0,4.5);
	\draw[thick,->] (0,-2) -- (0,-4.5);
	\draw[thick,->] (0,-2) -- (-2.5,-4.5);

	\draw (0,0) node[left] {$1$};
	\draw (1,1) node[right] {$\rho$};
	\draw (0,1) node[left] {$\tau$};
	\draw (0,-2) node[right] {$\theta$};
	\draw (-.9,-3.1) node[above left] {$\frac{\theta}{\rho}$};
	\draw (0,-3) node[right] {$\frac{\theta}{\tau}$};
\end{tikzpicture} 
\caption{$\Mt := H^{*,*}(\text{pt})$}
\label{M2}
\end{center}
\end{figure}

A $\text{Rep}(C_2)$-complex generalizes CW complexes using discs about the origin in representations. We denote these $D^{p,q}=D(\R^{p,q})$. The cohomology of a $\text{Rep}(C_2)$-complex can in principle be computed using a filtration spectral sequence whose first page is a free $\Mt$ module made up of one shifted copy of $\Mt$ corresponding to each cell. Although subsequent pages of the spectral sequence need not be free, we show in \cite{Us} that the cohomology of such spaces is always a free $\Mt$-module of the form $\bigoplus_i\Sigma^{a_i,b_i}\Mt$. We do this inductively, showing that if a $\text{Rep}(C_2)$ complex has a new cell attached, the resulting space may have some of its free summands `shifted' up in the sign grading, while the new summand corresponding to the new cell has shifted down. These are called Kronholm shifts after their first appearance in \cite{K}. See Example \ref{ex:rp2tw}.

Rather than displaying every copy of $\mathbb{F}_2$ in each bidegree, we will restrict ourselves to depicting the $\Mt$-generators with dots as in Figure \ref{fig:basicshift} below\footnote{Ignoring all of $\Mt$ besides the generator makes things easier to represent, but notice for example that while $E_1^{2,0}$ appears empty, $\Sigma^{2,2}\theta$ lives there, which is why a differential from $\Sigma^{1,0}\Mt$ is possible.} or with rank tables giving the number of generators in each bidegree, as in Figure \ref{fig:242}.\\

\begin{example} 
\label{ex:rp2tw}
A simple (and in some sense the fundamental) example is when $\R P^2_{\text{tw}}=\Gr_1(\R^{3,1})$ is built from a Rep-CW structure by identifying $\R^{3,1}$ as $\R^{++-}$. In computing the cohomology using this filtration spectral sequence, we see a Kronholm shift from $E_1=\Mt\oplus \Sigma^{1,0}\Mt\oplus \Sigma^{2,2}\Mt$ to the cohomology $H^{\ast,\ast}(\R P^2_{\text{tw}})=\Mt\oplus \Sigma^{1,1}\Mt\oplus \Sigma^{2,1}\Mt$. See Figure \ref{fig:basicshift}. We will sometimes describe $\Sigma^{1,0}\Mt$ as having ``shifted" up to $\Sigma^{1,1}\Mt$ while $\Sigma^{2,2}\Mt$ shifted down to $\Sigma^{2,1}\Mt$.
\end{example} 

\begin{figure}[ht]
	\begin{center}
	\begin{tikzpicture}[scale=.6]
		\draw (1,-.1)--(1,.1);
		\draw (2,-.1)--(2,.1);
		\draw (-.1,1)--(.1,1); 
		\draw (-.1,2)--(.1,2);
		\draw[->] (0,0)--(0,3);
		\draw[->] (0,0)--(3,0);
		\draw[thick,-] (4,1)--(5,1);
		\draw[thick,-] (4,1.2)--(5,1.2);
		\draw[thick,-] (5.1,1.1)--(4.8,1.4);
		\draw[thick,-] (5.1,1.1)--(4.8,0.8);
		\draw (0.1,0.1) node {$\bullet$};
		\draw[->] (1.1,.4)--(1.1,.9);
		\draw[->] (2.1,1.8)--(2.1,1.3);
		\color{teal}
		\draw (1.1,0.1) node {$\bullet$};
		\color{magenta}
		\draw (2.1,2.1) node {$\bullet$};
	\end{tikzpicture}
	\quad
	\begin{tikzpicture}[scale=.6]
		\draw (1,-.1)--(1,.1);
		\draw (2,-.1)--(2,.1);
		\draw (-.1,1)--(.1,1);
		\draw (-.1,2)--(.1,2);
		\draw[->] (0,0)--(0,3);
		\draw[->] (0,0)--(3,0);
		\draw (0.1,0.1) node {$\bullet$};
		\draw[->] (1.1,.4)--(1.1,.9);
		\draw[->] (2.1,1.8)--(2.1,1.3);
		\color{violet}
		\draw (1.1,1.1) node {$\bullet$};
		\color{blue}
		\draw (2.1,1.1) node {$\bullet$};
		
	\end{tikzpicture}
	\caption{$E_1 \Rightarrow H^{\ast,\ast}(\R P^2_{\text{tw}})$.}
	\end{center}
	\label{fig:basicshift}
\end{figure}

In Sections 2.3 and 2.4 of \cite{Me} we give a method for creating multiple distinct $\text{Rep}(C_2)$ constructions of (and hence multiple spectral sequences for) Grassmannian manifolds of the form $\Gr_k(V)$ where $V$ is some real representation of $C_2$. This is also discussed in Example 4.1 of \cite{Dan} and a little in Example \ref{ex:schubert} of this paper.

We will denote by $\mathcal{U}$ the forgetful functor from equivariant spaces to spaces, and by $\psi$ the corresponding map from equivariant cohomology to singular cohomology. Finally, we will use $\mathcal{F}$ to denote  the fixed-set functor on spaces. Note that $\mathcal {F}(S^{a,b})=S^{a-b}$ and if $H^{\ast,\ast}(X)=\bigoplus_i\Sigma^{a_i,b_i}\Mt$ then $H^\ast(\mathcal FX)=\bigoplus_i\Sigma^{a_i-b_i}\Ft$.

%%%%%%%%%%%%%%%%%%%%%%%%%%%%%%%%
%%%% COMPUTATIONAL TOOLS
%%%%%%%%%%%%%%%%%%%%%%%%%%%%%%%%

\section{Computational tools}
\label{section:tools}

	Given a non-equivariant space $X$, we have a graded $\mathbb{F}_2$-module $H^\ast (X;\Ft)$ and we can define a classical Poincar\'e polynomial in $\Ft$ coefficients, $$ P(H^\ast X)=\sum_i\rank_{\mathbb{F}_2}(H^i(X);\mathbb{F}_2)x^i.$$
	However since our theory is bigraded, we will define another polynomial statistic $\P$ in two variables for free $\Mt$-modules:
	$$\mathcal{P}\left(\bigoplus_i \Sigma^{a_i,b_i}\Mt\right)=\sum_ix^{a_i}y^{b_i}.$$
	The natural transformations $\psi$ and $\mathcal F$ can be reflected on the polynomial side with operators $U:f(x,y)\mapsto f(x,1)$ and $F:f(x,y)\mapsto f(x,x\inv)$ as shown in Figure \ref{fig:commsquares}.

\begin{figure}[ht]
	\begin{center}
	\begin{tikzpicture}[scale=.7,node distance=1.5cm, descr/.style={fill=white,inner sep=2pt}]
			\def\scale{4}
			\draw(0,0) node {$H^{\ast} (\mathcal U(\,\,))$};
			\draw(0,\scale) node {$\H\ast\ast (\,\,)$};
			\draw(\scale,0) node {$\Z[x]$};
			\draw(\scale,\scale) node {$\Z[x,y]$};
			\draw[->](1,\scale) to node[descr] {$\mathcal{P}$} (\scale-0.9,\scale);
			\draw[->](1.3,0) to node[descr] {$P$} (\scale-.7,0);
			\draw[->](0,\scale-0.5) to node[descr] {$\psi$} (0,0.5);
			\draw[->](\scale,\scale-0.5) to node[descr] {$U$} (\scale,0.5);
		\end{tikzpicture}
		\qquad
		\begin{tikzpicture}[scale=.7,node distance=2cm, descr/.style={fill=white,inner sep=2pt}]
				\def\scale{4}
				\draw(0,0) node {$H^{\ast} (\mathcal F(\,\,))$};
				\draw(0,\scale) node {$\H\ast\ast (\,\,)$};
				\draw(\scale,0) node {$\Z[x]$};
				\draw(\scale,\scale) node {$\Z[x,y]$};
				\draw[->](1,\scale) to node[descr] {$\mathcal{P}$} (\scale-0.9,\scale);
				\draw[->](1.4,0) to node[descr] {$P$} (\scale-.7,0);
				\draw[->](0,\scale-0.5) to node[descr] {$\circ\mathcal F$} (0,0.5);
				\draw[->](\scale,\scale-0.5) to node[descr] {$F$} (\scale,0.5);
		\end{tikzpicture}
	\caption{$U(\P(H^{\ast,\ast}X))=P(\psi(H^{\ast,\ast}X))$ and $F(\P(H^{\ast,\ast}X))=P(H^{\ast}\mathcal{F}(X))$.}
	\label{fig:commsquares}
	\end{center}
\end{figure}
\begin{defn}
If $A$ and $B$ are $\Mt$ modules, with $B$ obtained from $A$ by shifting an $\Mt$ up to a $\Sigma^{0,s}\Mt$ and a $\Sigma^{n,n+s}\Mt$ down to a $\Sigma^{n,n}\Mt$ as in a Kronholm shift\footnote{In other words, $A\oplus \Sigma^{0,s}\Mt\oplus \Sigma^{n,n}\Mt=B\oplus \Mt\oplus \Sigma^{n,n+s}\Mt$.} define the \emph{Kronholm polynomial} of the shift to be
\begin{align*}
	K_{n,s}&:=\mathcal{P}(B)-\mathcal{P}(A)\\
	&=x^0y^s-x^0y^0+x^ny^n-x^{n}y^{n+s}\\
	&=(1-x^ny^n)(y^s-1).
\end{align*}

\end{defn}

\begin{example} 
\label{ex:k11}
In Example \ref{ex:rp2tw} we saw a Kronholm shift from $A=\Mt\oplus {\color{teal}\Sigma^{1,0}\Mt}\oplus {\color{magenta}\Sigma^{2,2}\Mt}$ to the cohomology $B=\Mt\oplus {\color{violet}\Sigma^{1,1}\Mt}\oplus{\color{blue} \Sigma^{2,1}\Mt}$. See Figure \ref{fig:basicshift}. This shift has Kronholm polynomial $$\mathcal{P}(B)-\mathcal{P}(A)=(1+{\color{violet}xy}+{\color{blue}x^2y})-(1+{\color{teal}x}+{\color{magenta}x^2y^2})=x(1-xy)(y-1)=xK_{1,1}.$$

\end{example}

Notice that since the generator that shifts up by $s=1$ doesn't start as $\Mt$ but as {\color{teal}$\Sigma^{1,0}\Mt$}, and since the generator that shifts down is located $n=1$ topological dimensions to its right, the Kronholm polynomial of this shift is $x^1y^0K_{n,s}=xK_{1,1}$.\\
	
\begin{example}  
\label{ex:242}	
When $X=\Gr_2(\R^{4,2})$ is built by identifying $\R^{4,2}\iso \R^{++--}$, we have an $E_1$ page $E_1=\Mt\oplus\Sigma^{1,1}\Mt\oplus\Sigma^{2,1}\Mt^{\oplus 2}\oplus\Sigma^{3,1}\Mt\oplus\Sigma^{4,4}\Mt$ converging to $\H\ast\ast (X)=\Mt\oplus\Sigma^{1,1}\Mt\oplus\Sigma^{2,1}\Mt\oplus\Sigma^{2,2}\Mt\oplus\Sigma^{3,2}\Mt\oplus\Sigma^{4,2}\Mt$. See Figure \ref{fig:242}. We can also assign a Kronholm polynomial to this convergence, namely $$(x^4y^2 + x^3y^2 + x^2y^2 + x^2y + xy + 1 )-(x^4y^4 + x^3y + 2x^2y + xy + 1)
=x^3yK_{1,2}+x^2yK_{1,1}.$$

\end{example}

\begin{figure}[ht]
\begin{center}
	\begin{tikzpicture}[scale=.4]
		\draw (0,0)--(0,5);
		\draw[gray] (1,0)--(1,5);
		\draw[gray] (2,0)--(2,5);
		\draw (0,0)--(5,0);
		\draw[gray] (0,1)--(5,1);
		\draw[gray] (0,2)--(5,2);
		\draw[gray] (0,3)--(5,3);
		\draw[gray] (0,4)--(5,4);
		\draw[gray] (3,0)--(3,5);
		\draw[gray] (4,0)--(4,5);
		\draw(4,4) node[above right] {1};
		\draw(3,1) node[above right] {1};
		\draw(2,1) node[above right] {2};
		\draw(1,1) node[above right] {1};
		\draw(0,0) node[above right] {1};
		\draw(2.5,-1) node {$E_1$};
		\draw[thick,-] (3+3.5,1+1.3)--(5+3.5,1+1.3);
		\draw[thick,-] (3+3.5,1.2+1.3)--(5+3.5,1.2+1.3);
		\draw[thick,-] (5.1+3.5,1.1+1.3)--(4.8+3.5,1.4+1.3);
		\draw[thick,-] (5.1+3.5,1.1+1.3)--(4.8+3.5,0.8+1.3);
	\end{tikzpicture}
	\quad
	\begin{tikzpicture}[scale=.4]
		\draw (0,0)--(0,5);
		\draw[gray] (1,0)--(1,5);
		\draw[gray] (2,0)--(2,5);
		\draw (0,0)--(5,0);
		\draw[gray] (0,1)--(5,1);
		\draw[gray] (0,2)--(5,2);
		\draw[gray] (0,3)--(5,3);
		\draw[gray] (0,4)--(5,4);
		\draw[gray] (3,0)--(3,5);
		\draw[gray] (4,0)--(4,5);
		\draw(2.5,-1) node {$\H\ast\ast (X)$};
		\draw(4,2) node[above right] {1};
		\draw(3,2) node[above right] {1};
		\draw(2,2) node[above right] {1};
		\draw(2,1) node[above right] {1};
		\draw(1,1) node[above right] {1};
		\draw(0,0) node[above right] {1};
	\end{tikzpicture}
	\caption{Rank Table, counting copies of $\Mt$ in each bidegree}
	\label{fig:242}
\end{center}
\end{figure}

Since any change caused by a sequence of Kronholm shifts will show up in the image of $\mathcal P$ as a sum of Kronholm polynomials with coefficients from $\Z[x,y]$, we wish to study all such combinations, motivating the following definition.

\subsection{The Kronholm Ideal}

\begin{defn}
Let $\mathbb{K}\subseteq \Z[x,y]$ be the ideal generated by all Kronholm shifts $K_{n,s}$.\\
\end{defn}

\begin{comment}
An advantage of working in $\Z[x,y]$ rather than dealing directly with short exact or spectral sequences in cohomology is that order no longer matters -- we don't need to worry about what shifted first, or which generators on which pages supported the shift. However we also lose access to certain information, for example, when decomposing a Kronholm polynomial, we may be tempted to try to tell the ``story'' of which Kronholm shifts occurred, but in general there will be ambiguities, as the following lemmas demonstrate.\\
\end{comment}

\begin{example}
\label{ex:ambig1}
Suppose there is a sequence of shifts in which the ``downshifted'' generator from previous shifts becomes the ``upshifted'' generator in the next.\footnote{Say if we compute cohomology each time we attach a new cell to a complex.}

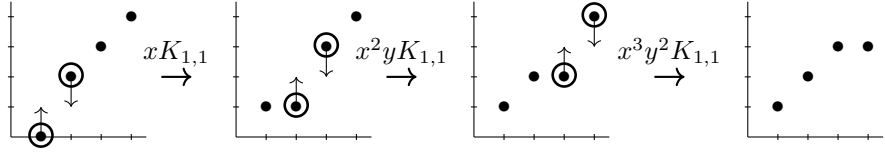
\begin{center}
	
	\begin{tikzpicture}[scale=.3]
		\draw(0,-.5) node{}; 
		\draw (0,0)--(0,4.5);
		\draw (0,0)--(4.5,0);
		\draw (2,-.1)--(2,.1);
		\draw (3,-.1)--(3,.1);
		\draw (4,-.1)--(4,.1);
		\draw (-.1,1)--(.1,1);
		\draw (-.1,2)--(.1,2);
		\draw (-.1,3)--(.1,3);
		\draw (-.1,4)--(.1,4);
		\draw (1,0) node {$\bullet$};
		\draw (2,2) node {$\bullet$};
		\draw (3,3) node {$\bullet$};
		\draw (4,4) node {$\bullet$};
%		\draw (3,2) node {\Huge$\circ$};
%		\draw (4,4) node {\Huge$\circ$};
		\draw[->] (1,0)--(1,1);
		\draw[->] (2,2)--(2,1);
		\draw[thick,->] (6,2)--node[above]{$xK_{1,1}$}(7,2);
	\end{tikzpicture}
	\begin{tikzpicture}[scale=.3]
		\draw(0,-.5) node{}; 
		\draw (0,0)--(0,4.5);
		\draw (0,0)--(4.5,0);
		\draw (1,-.1)--(1,.1);
		\draw (2,-.1)--(2,.1);
		\draw (3,-.1)--(3,.1);
		\draw (4,-.1)--(4,.1);
		\draw (-.1,1)--(.1,1);
		\draw (-.1,2)--(.1,2);
		\draw (-.1,3)--(.1,3);
		\draw (-.1,4)--(.1,4);
		\draw (1,1) node {$\bullet$};
		\draw (2,1) node {$\bullet$};
		\draw (3,3) node {$\bullet$};
		\draw (4,4) node {$\bullet$};
%		\draw (3,2) node {\Huge$\circ$};
%		\draw (4,4) node {\Huge$\circ$};
		\draw[->] (2,1)--(2,2);
		\draw[->] (3,3)--(3,2);
		\draw[thick,->] (6,2)--node[above]{$x^2yK_{1,1}$}(6,2);
	\end{tikzpicture}
	\begin{tikzpicture}[scale=.3]
		\draw(0,-.5) node{}; 
		\draw (0,0)--(0,4.5);
		\draw (0,0)--(4.5,0);
		\draw (1,-.1)--(1,.1);
		\draw (2,-.1)--(2,.1);
		\draw (3,-.1)--(3,.1);
		\draw (4,-.1)--(4,.1);
		\draw (-.1,1)--(.1,1);
		\draw (-.1,2)--(.1,2);
		\draw (-.1,3)--(.1,3);
		\draw (-.1,4)--(.1,4);
		\draw (1,1) node {$\bullet$};
		\draw (2,2) node {$\bullet$};
		\draw (3,2) node {$\bullet$};
		\draw (4,4) node {$\bullet$};
%		\draw (3,2) node {\Huge$\circ$};
%		\draw (4,4) node {\Huge$\circ$};
		\draw[->] (3,2)--(3,3);
		\draw[->] (4,4)--(4,3);
		\draw[thick,->] (7,2)--node[above]{$x^3y^2K_{1,1}$}(8,2);
	\end{tikzpicture}
	\begin{tikzpicture}[scale=.3]
		\draw(0,-.5) node{}; 
		\draw (0,0)--(0,4.5);
		\draw (0,0)--(4.5,0);
		\draw (1,-.1)--(1,.1);
		\draw (2,-.1)--(2,.1);
		\draw (3,-.1)--(3,.1);
		\draw (4,-.1)--(4,.1);
		\draw (-.1,1)--(.1,1);
		\draw (-.1,2)--(.1,2);
		\draw (-.1,3)--(.1,3);
		\draw (-.1,4)--(.1,4);
		\draw (1,1) node {$\bullet$};
		\draw (2,2) node {$\bullet$};
		\draw (3,3) node {$\bullet$};
		\draw (4,3) node {$\bullet$};
	\end{tikzpicture}
	
\end{center}

This sequence of shifts has the same outcome as a single shift involving only the first and last generator\footnote{Say if we just run one spectral sequence.}:

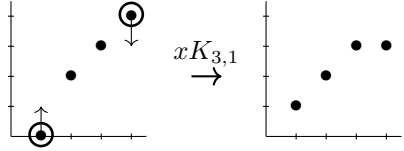
\begin{center}
	\begin{tikzpicture}[scale=.4]
		\draw(0,-.5) node{}; 
		\draw (0,0)--(0,4.5);
		\draw (0,0)--(4.5,0);
		\draw (2,-.1)--(2,.1);
		\draw (3,-.1)--(3,.1);
		\draw (4,-.1)--(4,.1);
		\draw (-.1,1)--(.1,1);
		\draw (-.1,2)--(.1,2);
		\draw (-.1,3)--(.1,3);
		\draw (-.1,4)--(.1,4);
		\draw (1,0) node {$\bullet$};
		\draw (2,2) node {$\bullet$};
		\draw (3,3) node {$\bullet$};
		\draw (4,4) node {$\bullet$};
%		\draw (3,2) node {\Huge$\circ$};
%		\draw (4,4) node {\Huge$\circ$};
		\draw[->] (1,0)--(1,1);
		\draw[->] (4,4)--(4,3);
		\draw[thick,->] (6,2)--node[above]{$xK_{3,1}$}(7,2);
	\end{tikzpicture}
	\begin{tikzpicture}[scale=.4]
		\draw(0,-.5) node{}; 
		\draw (0,0)--(0,4.5);
		\draw (0,0)--(4.5,0);
		\draw (1,-.1)--(1,.1);
		\draw (2,-.1)--(2,.1);
		\draw (3,-.1)--(3,.1);
		\draw (4,-.1)--(4,.1);
		\draw (-.1,1)--(.1,1);
		\draw (-.1,2)--(.1,2);
		\draw (-.1,3)--(.1,3);
		\draw (-.1,4)--(.1,4);
		\draw (1,1) node {$\bullet$};
		\draw (2,2) node {$\bullet$};
		\draw (3,3) node {$\bullet$};
		\draw (4,3) node {$\bullet$};
	\end{tikzpicture}
\end{center}

The algebraic fact shown, that $x(1+xy+x^2y^2)K_{1,1}=xK_{3,1}$, illustrates that knowing the beginning and end of the story may not be enough to determine that sequence of events between. This particular type of ambiguity is captured in the following lemma.
\end{example}

\begin{lemma}
	\label{lem1}
	$\ds K_{n,s}=\left(\sum_{i=0}^{n-1}(xy)^i\right)K_{1,s}$.
\end{lemma}
Algebraically this is immediate:\\ 
$$\left(\sum_{i=0}^{n-1}(xy)^i\right)K_{1,s}=\frac{1-(xy)^n}{1-xy}(1-xy)(y^s-1)=(1-x^ny^n)(y^s-1)=K_{n,1}.$$\\

\begin{example}
\label{ex:ambig2}
Similarly, we can imagine multiple shifts happened in sequence in the same topological dimensions, resulting in the same outcome as one ``larger'' shift. Figure \ref{fig:versus} illustrates\footnote{This example doesn't correspond to a possible scenario coming from equivariant spaces. If this bothers you, multiply everything by $x^3$ for a scenario where no bidegree's weight exceeds its topological dimension.} the algebraic fact that $K_{1,3}=(1+y+y^2)K_{1,1}$. This too can be generalized with a lemma.
 
\begin{figure}[ht!]
	\begin{center}
	\begin{tikzpicture}[scale=.4]
		\draw(0,-.5) node{}; 
		\draw (0,0)--(0,4.5);
		\draw (0,0)--(2,0);
		\draw (1,-.1)--(1,.1);
		\draw (-.1,1)--(.1,1);
		\draw (-.1,2)--(.1,2);
		\draw (-.1,3)--(.1,3);
		\draw (-.1,4)--(.1,4);
		\draw (0,0) node {$\bullet$};
		\draw (1,2) node {$\bullet$};
		\draw (1,3) node {$\bullet$};
		\draw (1,4) node {$\bullet$};
		\draw[->] (-.2,0)--(-.2,1);
		\draw[->] (1,2)--(1,1);
		\draw[thick,->] (3,2)--node[above]{$K_{1,1}$}(5,2);
	\end{tikzpicture}
	\begin{tikzpicture}[scale=.4]
		\draw(-1,-.5) node{}; 
		\draw (0,0)--(0,4.5);
		\draw (0,0)--(2,0);
		\draw (1,-.1)--(1,.1);
		\draw (-.1,1)--(.1,1);
		\draw (-.1,2)--(.1,2);
		\draw (-.1,3)--(.1,3);
		\draw (-.1,4)--(.1,4);
		\draw (0,1) node {$\bullet$};
		\draw (1,1) node {$\bullet$};
		\draw (1,3) node {$\bullet$};
		\draw (1,4) node {$\bullet$};
		\draw[->] (-.2,1)--(-.2,2);
		\draw[->] (1,3)--(1,2);
		\draw[thick,->] (3,2)--node[above]{$yK_{1,1}$}(5,2);
	\end{tikzpicture}
	\begin{tikzpicture}[scale=.4]
		\draw(-1,-.5) node{}; 
		\draw (0,0)--(0,4.5);
		\draw (0,0)--(2,0);
		\draw (1,-.1)--(1,.1);
		\draw (-.1,1)--(.1,1);
		\draw (-.1,2)--(.1,2);
		\draw (-.1,3)--(.1,3);
		\draw (-.1,4)--(.1,4);
		\draw (0,2) node {$\bullet$};
		\draw (1,1) node {$\bullet$};
		\draw (1,2) node {$\bullet$};
		\draw (1,4) node {$\bullet$};
		\draw[->] (-.2,2)--(-.2,3);
		\draw[->] (1,4)--(1,3);
		\draw[thick,->] (3,2)--node[above]{$y^2K_{1,1}$}(5,2);
	\end{tikzpicture}
	\begin{tikzpicture}[scale=.4]
		\draw(-1,-.5) node{}; 
		\draw (0,0)--(0,4.5);
		\draw (0,0)--(2,0);
		\draw (0,3) node {$\bullet$};
		\draw (1,1) node {$\bullet$};
		\draw (1,2) node {$\bullet$};
		\draw (1,3) node {$\bullet$};
	\end{tikzpicture}\\
	versus\\
	\begin{tikzpicture}[scale=.4]
		\node at (1,4) (onefour) {};
		\node at (1,1) (oneone) {};
		\draw(0,-.5) node{};
		\draw (0,0)--(0,4.5);
		\draw (0,0)--(2,0);
		\draw (1,-.1)--(1,.1);
		\draw (-.1,1)--(.1,1);
		\draw (-.1,2)--(.1,2);
		\draw (-.1,3)--(.1,3);
		\draw (-.1,4)--(.1,4);
		\draw (0,0) node {$\bullet$};
		\draw (1,2) node {$\bullet$}; 
		\draw (1,3) node {$\bullet$};
		\draw (1,4) node {$\bullet$};
		\draw[->] (-.2,0)--(-.2,3);
		\draw[->, bend left] (onefour) to node {} (oneone);
		\draw[thick,->] (4,2)--node[above]{$K_{1,3}$}(6,2);
	\end{tikzpicture}
	\begin{tikzpicture}[scale=.4]
		\draw(2,-.5) node{}; 
		\draw (4,0)--(4,4.5);
		\draw (4,0)--(6,0);
		\draw (5,-.1)--(5,.1);
		\draw (4-.1,1)--(4.1,1);
		\draw (4-.1,2)--(4.1,2);
		\draw (4-.1,3)--(4.1,3);
		\draw (4-.1,4)--(4.1,4);
		\draw (4,3) node {$\bullet$};
		\draw (5,2) node {$\bullet$};
		\draw (5,3) node {$\bullet$};
		\draw (5,1) node {$\bullet$};
	\end{tikzpicture}
	\caption{Two shifting scenarios with the same outcome.}
	\label{fig:versus}
	\end{center}
\end{figure}
\end{example}

\begin{lemma}
	\label{lem2}
	$\ds K_{1,s}=\left(\sum_{j=0}^{s-1}y^j\right)K_{1,1}$.\\
% Algebraically this is immediate:\\ $\ds\left(\sum_{j=0}^{s-1}y^j\right)K_{1,1}=\frac {y^s-1}{y-1}(1-xy)(y-1)=(1-xy)(y^s-1)=K_{1,s}$.
\end{lemma}

In fact, all Kronholm shifts relate back to the fundamental shift $K_{1,1}$.\\

\begin{thm}
	$\mathbb K$ is principally generated: $\mathbb K=(K_{1,1})$.
\end{thm}

\begin{proof}[Proof 1]
	Note that because Kronholm shifts are invisible to the natural transformations $F$ and $U$, $\mathbb K\subseteq \ker(F)\cap \ker(U)=(1-xy)\cap(y-1)=(K_{1,1})$ and of course $K_{1,1}\in \mathbb K$.
\end{proof}

\begin{proof}[Proof 2]
	It follows from the lemmas \ref{lem1} and \ref{lem2} that 
	$$ K_{n,s}=\left(\sum_{i=0}^{n-1}(xy)^i\right)\left(\sum_{j=0}^{s-1}y^j\right)K_{1,1}.$$
\end{proof}

We can now justify the following definition.\\

\begin{defn}
	\label{def:story}
	For $\Mt$-modules $A$ and $B$ related by Kronholm shifts, the \textbf{shift story} from $A$ to $B$ is the polynomial $$f(x,y)=\frac{\mathcal{P}(B)-\mathcal{P}(A)}{K_{1,1}}.$$
\end{defn}

\begin{example}
	The Kronholm shift in Example \ref{ex:k11} corresponds to a shift story of $f(x,y)=x=x^1y^0$, indicating a differential from $\Sigma^{1,0}\Mt$ causing it to shift up.
	
	But as the ambiguities in Examples \ref{ex:ambig1} and \ref{ex:ambig2} illustrate, this story is only a story. The shift story for $\Gr_2(\R^{4,2})$ in Example \ref{ex:242} is $f(x,y)=x^2y+x^3y+x^3y^2$. However $E_1$ has no generators in degree $(3,2)$. The appearance of the term $x^3y^2$ is an artifact of the identity $x^3yK_{1,2}=x^3y(1+y)K_{1,1}$ from Lemma \ref{lem2}.\\
\end{example}

A reader interested only in computing Bredon cohomology may safely skip the following section.

\subsection{Aside: Another invariant}

\begin{comment}
		We could define a coarser invariant on $\text{Rep}(C_2)$-complexes $X$ given by the image of $\mathcal{P}(\H\ast\ast (X))$ under the quotient $\Z[x,y]\to \Z[x,y]/\mathbb{K}$. This is still not as coarse as the traditional Poincar\'e polynomial of the underlying space. For $n$-dimensional manifolds we can choose representatives for these equivalence classes to take the form $[p(x)+x^ny^T]$ where $p$ has degree $n-1$, and we can speak of the \emph{total weight} $T$ associated to a space. For instance, looking back at Example \ref{ex:k11},  
		$$[\mathcal{P}(\H\ast\ast (\RP^2_{tw}))]=[1+xy+x^2y]=[(1+x)+x^2y^2]$$ 
		giving total weight $T(\RP^2_{tw})=2$ while from Example \ref{ex:242}, 
		$$[\mathcal{P}(\H\ast\ast (\Gr_2\R^{4,2}))]=[(1+x+2x^2+x^3)+x^4y^8]$$ and so $T(\Gr_2\R^{4,2})=8$. 
		
		This invariant $T$ is Kronholm-shift agnostic. For Grassmannians, it can be computed without running a spectral sequence since it can be read off of the first page by simply adding up weights. Playing with some computed examples and the \cite{oeis} leads to the formula in Theorem \ref{TW} for these spaces' total weights.

\end{comment}

Before stating the theorem, we recall that each Schubert cell can be indexed by a Young diagram. The cell's dimension corresponds to the number of degrees of freedom in a matrix.\\

\begin{example}
	\label{ex:schubert}
	
	 For example $\Gr_3\R^7$ contains the cell 
	
	\[\Omega_{\text{\tiny$\yng(1,3,4)$}}=\left\{\text{rowspace}\begin{bmatrix}a&1&0&0&0&0&0\\b&0&c&d&1&0&0\\e&0&f&g&0&h&1\end{bmatrix}:a,b,c,d,e,f,g,h\in \R\right\}\]

with eight degrees of freedom. Choosing for example the representation $\R^{--++-++}$ we have the action

\begin{align*}
	\text{rowspace}\begin{bmatrix}a&1&0&0&0&0&0\\b&0&c&d&1&0&0\\e&0&f&g&0&h&1\end{bmatrix}
	&\mapsto
	\text{rowspace}\begin{bmatrix}-a&-1&0&0&0&0&0\\-b&0&c&d&-1&0&0\\-e&0&f&g&0&h&1\end{bmatrix}\\
	&=\text{rowspace}\begin{bmatrix}a&1&0&0&0&0&0\\b&0&-c&-d&1&0&0\\-e&0&f&g&0&h&1\end{bmatrix}
\end{align*}

Or more briefly, {\small $\young(+,+++,++++)\mapsto \young(+,+--,-+++)$}.

In this construction, this particular Schubert cell is homeomorphic to $D^{8,3}$.\\

\end{example}

For a given degree of freedom in a given Schubert cell, consider its horizontal position in the Schubert matrix as well as the last nonzero horizontal position in that row. This degree of freedom will have a sign action when one or the other of these horizontal positions corresponds to an $\R^-$ (but not both). For example note the degree of freedom $c$  in the Schubert cell above has a sign action not just for $\R^{--++-++}$ but for any representation of the form $\R^{??\pm?\mp??}$. We are now ready to state and prove\\

\begin{thm}\label{TW}
	\[T(\Gr_k(\R^{p,q}))=(p-q)q\binom{p-2}{k-1}.\]
\end{thm}

\begin{proof}

	We will deliberately overcount, summing degrees of freedom with nontrivial action across all $\binom pq$ ``sprinklings,'' that is all of the ways of ordering $q$ copies of $\R^-$ among $p-q$ copies of $\R^+$ to build $\R^{p,q}$. The $\binom pk$ different Schubert cells have on average $\frac{k(p-k)}2$ degrees of freedom by symmetry of Young diagrams fitting with a $k\times (p-k)$ box. Each degree of freedom in each Schubert cell will have a sign action when horizontal position and that if its row's last nonzero entry correspond to different representations, as discussed above. The remaining $p-2$ horizontal positions may be occupied by the remaining $q-1$ sign representation in $\binom{p-2}{q-1}$ ways. And so summing over all sprinklings $s$ and noting that total weight is independent of this choice of $\R^{p,q}$, we have
	\[\sum_sT(s)=\binom pq T=\underbrace{\binom pk}_{\text{\# cells}}\cdot \underbrace{\frac{k(p-k)}2}_{\text{avg cell dim}}\cdot \underbrace{2\binom{p-2}{q-1}}_{\text{\#} \R^{-} \text{ actions}}\]
	And thus 
	\[T=\binom pk k(p-k)\cdot \frac{ \binom{p-2}{q-1}}{\binom pq}=p(p-1)\binom{p-2}{k-2}\cdot \frac{(p-q)q}{p(p-1)}=(p-q)q\binom{p-2}{k-2}.\]
\end{proof}

\begin{example} Suppose we want to check $T(\Gr_1\R^{3,2})$. There are $\binom32$ ways to identify $\R^{3,2}$.\\
\begin{figure}[ht]
	\begin{center}
		\begin{tabular}{r|c|c|c}
			&$\R^{--+}$&$\R^{-+-}$&$\R^{+--}$\\
			\hline
			\phantom{$\displaystyle\int$}2-cell&\young(--)&\young(+-)&\young(-+)\\
			\phantom{$\displaystyle\int$}1-cell&\young(+)&\young(-)&\young(-)\\
			\phantom{$\displaystyle\int$}0-cell&*&*&*
		\end{tabular}
	\end{center}
	\caption{Three ways to construct the same space}
	\label{threeways}
\end{figure}
We count $\binom32T=\underbrace{\binom31}_{\text{\# cells}}\cdot \underbrace{\frac22}_{\text{avg  dim}}\underbrace{2\binom11}_{\# -}=3\cdot 1\cdot 2$. the last term indicates that any given degree of freedom will be signed exactly twice, as can be checked by reading horizontally in Figure \ref{threeways}.\\
\end{example}

\begin{comment}
	This combinatorial proof tells a story that still requires some manipulation to produce our simpler formula. It would be nice to learn that the formula itself can be more directly interpreted combinatorially.
\end{comment}

This ends the aside about $\Z[x,y]/\mathbb K$. We next explain an algorithm to use the ideal $\mathbb K$ to efficiently compute some equivariant spectral sequences for cohomology.

%%%%%%%%%%%%%%%%%%%%%%%%%%%%%%%%
%%%% ALGORITHM DESIGN
%%%%%%%%%%%%%%%%%%%%%%%%%%%%%%%%

\section{Algorithm Design}
\label{section:alg}

A naive approach to computing $\H\ast\ast\Gr_k\R^{p,q}$ might look like

\subsection{Naive algorithm}
\begin{itemize}
	\item Calculate, for each of the $\binom pq$ ways of identifying $\R^{p,q}$ as an ordered direct sum of one-dimensional representations, the $E_1$ page of the filtration spectral sequence for $\Gr_k\R^{p,q}$.
	\item For the $i^\text{th}$ $E_1$ page, find the set $D_i$ of possible differentials, based on bidegree. Since each could be zero or nonzero, this gives roughly $2^{|D_i|}$ possible $E_\infty$ pages, each an associate graded of a free $\Mt$ module. Calculate each of these, creating a set $P_i$ of free modules to which the $i^{\text{th}}$ spectral sequence could possibly converge. Repeat for $1\le i \le \binom pq$. 
	\item Calculate the common set $\bigcap P_i$ of possible answers. If this is a singleton set, we have our cohomology.
\end{itemize}

While practicable for small examples, this approach is computationally expensive. Instead, we use the ideas of the previous section for some refinements.

Define the \textbf{tension} of a free $\Mt$ module $A$ be the number $\mathcal{P}(A)(1,2)$. Since $K_{1,1}(1,2)=-1$, any sequence of Kronholm shifts will reduce tension. This gives an easily computable heuristic. A ``relaxed" $E_1$ page with the lower tension should intuitively have fewer possible $E_\infty$ pages, as compared to a higher-tension $E_1$ which has more potential Kronholm shifts.

Given two free $\Mt$ modules $A$ and $B$, we can easily answer the question of whether $B$ is a Kronholm shift of $A$. First, the polynomial $\mathcal{P}(B)-\mathcal{P}(A)$ must lie in $\mathbb K$. But further it must be the case that the \textbf{shift story} $f(x,y)=\frac{\mathcal{P}(B)-\mathcal{P}(A)}{K_{1,1}}$ is a polynomial with strictly positive coefficients, so we are doing (rather than undoing) shifts to $A$ in order to get $B$. This can be quickly checked with a computer algebra system.

With these tools, we execute the following steps.

\subsection{Improved algorithm}

\begin{enumerate}[a.]
	\item Calculate, for each of the $\binom pq$ ways of identifying $\R^{p,q}$ as a direct sum of one-dimensional representations, the $E_1$ page of the filtration spectral sequence for $\Gr_k\R^{p,q}$. Discard any duplicates.
	\item Find the page with the lowest tension. Call it the $r^\text{th}$ $E_1$ page. Find the set of $D_r$ of possible differentials and compute all of the roughly $2^{|D_r|}$ free modules to which this page could possibly converge.\footnote{This part is still slow, despite the savings gained by choosing a low-tension starting point, with comparatively small $D$ value.} Discard duplicates. We now have a set $P_r$ of possible answers starting from the most relaxed construction. 
	\item For each $i\ne r$ we have an $E_1$ page that must converge to the right answer. Any candidate answer in $P_r$ to which some other $E_1$ page \emph{cannot} converge is therefore incorrect.\par
	But first we reduce redundancy. If the shift story indicates that the $i^\text{th}$ $E_1$ page can shift to the $j^\text{th}$ $E_1$ page, discard the former, as the latter has a more restrictive set of possible outcomes.
	\item For each remaining $E_1$ page and each possible answer $A\in P_r$, use the shift story $\frac{\mathcal{P}(A)-\mathcal{P}(E_1)}{K_{1,1}}$ to check if $A$ is a Kronholm shift of $E_1$. If not, discard $A$.
	\item The list of un-discarded elements of $P_r$ is our set of possible answers. With any luck, $|P_r|=1$.\\
\end{enumerate} 

\begin{example}
\label{ex363}
	
If we wish to compute $\H\ast\ast\Gr_3\R^{6,3}$, the steps look like:

\begin{enumerate}[a.]
	\item Of the $\binom 63=20$ possible constructions, there are 6 unique $E_1$ pages.
	\item One of these has Poincar\'e polynomial $p_r(x,y)=x^9y^5 + x^8y^4 + 2x^7y^4$
	\[ + x^6y^4 + x^5y^5 + 2x^6y^3 + 2x^5y^3 + x^4y^4 + 2x^4y^2 + 2x^3y^2 + x^3y + 2x^2y + xy + 1\]
	with a tension of $p_r(1,2)=201$, the lowest value of the six. After a few dozen milliseconds, we can list the 24 unique free modules associated to all possible $E_\infty$ pages for this $E_1$.
	\item Three of the remaining five $E_1$ pages can shift to others, leaving only two non-redundant starting points.
	\item One of these eliminates 11 of our possible answers since they are not valid Kronholm shifts. The next $E_1$ page eliminates 7 more.
	\item Unfortunately this still leaves $|P_r|=6$. While quite a few new calculations can be successfully executed in this way (see Section \ref{results}) in general still more work is required. In particular, for the completion of this calculation of $\H\ast\ast\Gr_3\R^{6,3}$ using some extra tricks, see Section \ref{finishing363}.
\end{enumerate}

\end{example}

%%%%%%%%%%%%%%%%%%%%%%%%%%%%%%%%
%%%% RESULTS
%%%%%%%%%%%%%%%%%%%%%%%%%%%%%%%%

\section{Results}
\label{results}

Here we summarize new and previously known results for $\H\ast\ast(\Gr_k\R^{p,q})$. 

Note that because $\perp:\Gr_k\R^{p,q}\to\Gr_{p-k}\R^{p,q}$ is an isomorphism and $\Gr_k\R^{p,q}\iso\Gr_k\R^{p,p-q}$, we can restrict to $k,q\le\frac p2$. A formula for $\H\ast\ast(\Gr_1\R^{p,q})=\mathbb{P}(\R^{p,q})$ appears in \cite{K}.

\subsection{$q=1$}

A formula for $\H\ast\ast(\Gr_k\R^{p,1})$. appears in \cite{Me}.

\subsection{$q=2$}

A formula for $\H\ast\ast(\Gr_2\R^{p,2})$. appears in \cite{Me}. The algorithm from Section \ref{section:alg} also yields $\H\ast\ast(\Gr_3\R^{6,2})$. See Figure \ref{fig:362}.

\begin{figure}[ht!]
	\begin{center}
	\begin{tikzpicture}[scale=.4]
		\draw (0,0)--(0,5);
		\draw (0,0)--(10,0);
		\draw[gray] (1,0)--(1,5);
		\draw[gray] (2,0)--(2,5);
		\draw[gray] (3,0)--(3,5);
		\draw[gray] (4,0)--(4,5);
		\draw[gray] (5,0)--(5,5);
		\draw[gray] (6,0)--(6,5);
		\draw[gray] (7,0)--(7,5);
		\draw[gray] (8,0)--(8,5);
		\draw[gray] (9,0)--(9,5);
		\draw[gray] (0,1)--(10,1);
		\draw[gray] (0,2)--(10,2);
		\draw[gray] (0,3)--(10,3);
		\draw[gray] (0,4)--(10,4);
		\draw(9,4) node[above right] {1};
		\draw(8,4) node[above right] {1};
		\draw(7,4) node[above right] {1};
		\draw(7,3) node[above right] {1};
		\draw(6,3) node[above right] {3};
		\draw(5,3) node[above right] {2};
		\draw(5,2) node[above right] {1};
		\draw(4,2) node[above right] {3};
		\draw(3,2) node[above right] {3};
		\draw(2,2) node[above right] {1};
		\draw(2,1) node[above right] {1};
		\draw(1,1) node[above right] {1};
		\draw(0,0) node[above right] {1};
	\end{tikzpicture}
	\vspace{-.1in}
	\caption{$\H\ast\ast(\Gr_3\R^{6,2})$}
	\label{fig:362}
	\end{center}
\end{figure}

However two possibilities remain for $\H\ast\ast(\Gr_3\R^{7,2})$. Part (b) of the computation for $\H\ast\ast(\Gr_4\R^{8,2})$ takes almost two minutes and leaves six possibilities. Further increases in $k$ and $p$ only make matters worse.

\subsection{$q=3$}

Figures \ref{fig:263} through \ref{fig:283} give $\H\ast\ast(\Gr_2\R^{6,3})$, $\H\ast\ast(\Gr_2\R^{7,3})$ and $\H\ast\ast(\Gr_2\R^{8,3})$.

\begin{figure}[ht!]
	\begin{center}
	\begin{tikzpicture}[scale=.4]
		\draw (0,0)--(0,5);
		\draw (0,0)--(9,0);
		\draw[gray] (1,0)--(1,5);
		\draw[gray] (2,0)--(2,5);
		\draw[gray] (3,0)--(3,5);
		\draw[gray] (4,0)--(4,5);
		\draw[gray] (5,0)--(5,5);
		\draw[gray] (6,0)--(6,5);
		\draw[gray] (7,0)--(7,5);
		\draw[gray] (8,0)--(8,5);
		\draw[gray] (0,1)--(9,1);
		\draw[gray] (0,2)--(9,2);
		\draw[gray] (0,3)--(9,3);
		\draw[gray] (0,4)--(9,4);
		\draw(8,4) node[above right] {1};
		\draw(7,4) node[above right] {1};
		\draw(6,4) node[above right] {1};
		\draw(6,3) node[above right] {1};
		\draw(5,3) node[above right] {2};
		\draw(4,3) node[above right] {1};
		\draw(4,2) node[above right] {2};
		\draw(3,2) node[above right] {2};
		\draw(2,2) node[above right] {1};
		\draw(2,1) node[above right] {1};
		\draw(1,1) node[above right] {1};
		\draw(0,0) node[above right] {1};
	\end{tikzpicture}
	\vspace{-.2in}
	\end{center}
	\caption{$\H\ast\ast(\Gr_2\R^{6,3})$}
	\label{fig:263}
\end{figure}

\begin{figure}[ht!]
	\begin{center}
	\begin{tikzpicture}[scale=.4]
		\draw (0,0)--(0,6);
		\draw (0,0)--(11,0);
		\draw[gray] (1,0)--(1,6);
		\draw[gray] (2,0)--(2,6);
		\draw[gray] (3,0)--(3,6);
		\draw[gray] (4,0)--(4,6);
		\draw[gray] (5,0)--(5,6);
		\draw[gray] (6,0)--(6,6);
		\draw[gray] (7,0)--(7,6);
		\draw[gray] (8,0)--(8,6);
		\draw[gray] (9,0)--(9,6);
		\draw[gray] (10,0)--(10,6);
		\draw[gray] (0,1)--(11,1);
		\draw[gray] (0,2)--(11,2);
		\draw[gray] (0,3)--(11,3);
		\draw[gray] (0,4)--(11,4);
		\draw[gray] (0,5)--(11,5);
		\draw(10,5) node[above right] {1};
		\draw(9,5) node[above right] {1};
		\draw(8,4) node[above right] {2};
		\draw(7,4) node[above right] {2};
		\draw(6,4) node[above right] {1};
		\draw(6,3) node[above right] {2};
		\draw(5,3) node[above right] {3};
		\draw(4,3) node[above right] {1};
		\draw(4,2) node[above right] {2};
		\draw(3,2) node[above right] {2};
		\draw(2,2) node[above right] {1};
		\draw(2,1) node[above right] {1};
		\draw(1,1) node[above right] {1};
		\draw(0,0) node[above right] {1};
	\end{tikzpicture}
	\vspace{-.2in}
	\end{center}
	\caption{$\H\ast\ast(\Gr_2\R^{7,3})$}
\end{figure}

\begin{figure}[ht!]
	\begin{center}
	\begin{tikzpicture}[scale=.4]
		\draw (0,0)--(0,7);
		\draw (0,0)--(13,0);
		\draw[gray] (1,0)--(1,7);
		\draw[gray] (2,0)--(2,7);
		\draw[gray] (3,0)--(3,7);
		\draw[gray] (4,0)--(4,7);
		\draw[gray] (5,0)--(5,7);
		\draw[gray] (6,0)--(6,7);
		\draw[gray] (7,0)--(7,7);
		\draw[gray] (8,0)--(8,7);
		\draw[gray] (9,0)--(9,7);
		\draw[gray] (10,0)--(10,7);
		\draw[gray] (11,0)--(11,7);
		\draw[gray] (12,0)--(12,7);
		\draw[gray] (0,1)--(13,1);
		\draw[gray] (0,2)--(13,2);
		\draw[gray] (0,3)--(13,3);
		\draw[gray] (0,4)--(13,4);
		\draw[gray] (0,5)--(13,5);
		\draw[gray] (0,6)--(13,6);
		\draw(12,6) node[above right] {1};
		\draw(11,5) node[above right] {1};
		\draw(10,5) node[above right] {2};
		\draw(9,5) node[above right] {1};
		\draw(9,4) node[above right] {1};
		\draw(8,4) node[above right] {3};
		\draw(7,4) node[above right] {2};
		\draw(7,3) node[above right] {1};
		\draw(6,4) node[above right] {1};
		\draw(6,3) node[above right] {3};
		\draw(5,3) node[above right] {3};
		\draw(4,3) node[above right] {1};
		\draw(4,2) node[above right] {2};
		\draw(3,2) node[above right] {2};
		\draw(2,2) node[above right] {1};
		\draw(2,1) node[above right] {1};
		\draw(1,1) node[above right] {1};
		\draw(0,0) node[above right] {1};
	\end{tikzpicture}
	\vspace{-.2in}
	\end{center}
	\caption{$\H\ast\ast(\Gr_2\R^{8,3})$}
	\label{fig:283}
\end{figure}

$\H\ast\ast(\Gr_2\R^{9,3})$ has 2 possibilities. 

For $\H\ast\ast(\Gr_3\R^{6,3})$ see Example \ref{ex363} and Section \ref{finishing363}. 

Brute-forcing all possible $E_\infty$ pages for even the most relaxed $E_1$ page of $\H\ast\ast(\Gr_4\R^{8,3})$ appears to be beyond the computing resources readily available to us.

\subsection{$q=4$}

\begin{figure}[ht!]
	\begin{center}
	\begin{tikzpicture}[scale=.4]
		\draw (0,0)--(0,7);
		\draw (0,0)--(13,0);
		\draw[gray] (1,0)--(1,7);
		\draw[gray] (2,0)--(2,7);
		\draw[gray] (3,0)--(3,7);
		\draw[gray] (4,0)--(4,7);
		\draw[gray] (5,0)--(5,7);
		\draw[gray] (6,0)--(6,7);
		\draw[gray] (7,0)--(7,7);
		\draw[gray] (8,0)--(8,7);
		\draw[gray] (9,0)--(9,7);
		\draw[gray] (10,0)--(10,7);
		\draw[gray] (11,0)--(11,7);
		\draw[gray] (12,0)--(12,7);
		\draw[gray] (0,1)--(13,1);
		\draw[gray] (0,2)--(13,2);
		\draw[gray] (0,3)--(13,3);
		\draw[gray] (0,4)--(13,4);
		\draw[gray] (0,5)--(13,5);
		\draw[gray] (0,6)--(13,6);
		\draw(12,6) node[above right] {1};
		\draw(11,6) node[above right] {1};
		\draw(10,6) node[above right] {1};
		\draw(10,5) node[above right] {1};
		\draw(9,5) node[above right] {2};
		\draw(8,5) node[above right] {1};
		\draw(8,4) node[above right] {2};
		\draw(7,4) node[above right] {3};
		\draw(6,4) node[above right] {2};
		\draw(6,3) node[above right] {2};
		\draw(5,3) node[above right] {3};
		\draw(4,3) node[above right] {1};
		\draw(4,2) node[above right] {2};
		\draw(3,2) node[above right] {2};
		\draw(2,2) node[above right] {1};
		\draw(2,1) node[above right] {1};
		\draw(1,1) node[above right] {1};
		\draw(0,0) node[above right] {1};
	\end{tikzpicture}
	\vspace{-.2in}
	\end{center}
	\caption{$\H\ast\ast(\Gr_2\R^{8,4})$}
\end{figure}

\begin{figure}[ht!]
	\begin{center}
	\begin{tikzpicture}[scale=.4]
		\draw (0,0)--(0,8);
		\draw (0,0)--(15,0);
		\draw[gray] (1,0)--(1,8);
		\draw[gray] (2,0)--(2,8);
		\draw[gray] (3,0)--(3,8);
		\draw[gray] (4,0)--(4,8);
		\draw[gray] (5,0)--(5,8);
		\draw[gray] (6,0)--(6,8);
		\draw[gray] (7,0)--(7,8);
		\draw[gray] (8,0)--(8,8);
		\draw[gray] (9,0)--(9,8);
		\draw[gray] (10,0)--(10,8);
		\draw[gray] (11,0)--(11,8);
		\draw[gray] (12,0)--(12,8);
		\draw[gray] (13,0)--(13,8);
		\draw[gray] (14,0)--(14,8);
		\draw[gray] (0,1)--(15,1);
		\draw[gray] (0,2)--(15,2);
		\draw[gray] (0,3)--(15,3);
		\draw[gray] (0,4)--(15,4);
		\draw[gray] (0,5)--(15,5);
		\draw[gray] (0,6)--(15,6);
		\draw[gray] (0,7)--(15,7);
		\draw(14,7) node[above right] {1};
		\draw(13,7) node[above right] {1};
		\draw(12,6) node[above right] {2};
		\draw(11,6) node[above right] {2};
		\draw(10,6) node[above right] {1};
		\draw(10,5) node[above right] {2};
		\draw(9,5) node[above right] {3};
		\draw(8,5) node[above right] {1};
		\draw(8,4) node[above right] {3};
		\draw(7,4) node[above right] {4};
		\draw(6,4) node[above right] {2};
		\draw(6,3) node[above right] {2};
		\draw(5,3) node[above right] {3};
		\draw(4,3) node[above right] {1};
		\draw(4,2) node[above right] {2};
		\draw(3,2) node[above right] {2};
		\draw(2,2) node[above right] {1};
		\draw(2,1) node[above right] {1};
		\draw(1,1) node[above right] {1};
		\draw(0,0) node[above right] {1};
	\end{tikzpicture}
	\vspace{-.2in}
	\end{center}
	\caption{$\H\ast\ast(\Gr_2\R^{9,4})$}
\end{figure}

\newpage
\subsection{$q=5$}

\begin{figure}[ht!]
	\begin{center}
	\begin{tikzpicture}[scale=.4]
		\draw (0,0)--(0,9);
		\draw (0,0)--(17,0);
		\draw[gray] (1,0)--(1,9);
		\draw[gray] (2,0)--(2,9);
		\draw[gray] (3,0)--(3,9);
		\draw[gray] (4,0)--(4,9);
		\draw[gray] (5,0)--(5,9);
		\draw[gray] (6,0)--(6,9);
		\draw[gray] (7,0)--(7,9);
		\draw[gray] (8,0)--(8,9);
		\draw[gray] (9,0)--(9,9);
		\draw[gray] (10,0)--(10,9);
		\draw[gray] (11,0)--(11,9);
		\draw[gray] (12,0)--(12,9);
		\draw[gray] (13,0)--(13,9);
		\draw[gray] (14,0)--(14,9);
		\draw[gray] (15,0)--(15,9);
		\draw[gray] (16,0)--(16,9);
		\draw[gray] (0,1)--(17,1);
		\draw[gray] (0,2)--(17,2);
		\draw[gray] (0,3)--(17,3);
		\draw[gray] (0,4)--(17,4);
		\draw[gray] (0,5)--(17,5);
		\draw[gray] (0,6)--(17,6);
		\draw[gray] (0,7)--(17,7);
		\draw[gray] (0,8)--(17,8);
		\draw(16,8) node[above right] {1};
		\draw(15,8) node[above right] {1};
		\draw(14,8) node[above right] {1};
		\draw(14,7) node[above right] {1};
		\draw(13,7) node[above right] {2};
		\draw(12,7) node[above right] {1};
		\draw(12,6) node[above right] {2};
		\draw(11,6) node[above right] {3};
		\draw(10,6) node[above right] {2};
		\draw(10,5) node[above right] {2};
		\draw(9,5) node[above right] {4};
		\draw(8,5) node[above right] {2};
		\draw(8,4) node[above right] {3};
		\draw(7,4) node[above right] {4};
		\draw(6,4) node[above right] {2};
		\draw(6,3) node[above right] {2};
		\draw(5,3) node[above right] {3};
		\draw(4,3) node[above right] {1};
		\draw(4,2) node[above right] {2};
		\draw(3,2) node[above right] {2};
		\draw(2,2) node[above right] {1};
		\draw(2,1) node[above right] {1};
		\draw(1,1) node[above right] {1};
		\draw(0,0) node[above right] {1};
	\end{tikzpicture}
	% \vspace{-.2in}
	\end{center}
	\caption{$\H\ast\ast(\Gr_2\R^{10,5})$}
\end{figure}

\begin{figure}[ht!]
	\begin{center}
	\begin{tikzpicture}[scale=.4]
		\draw (0,0)--(0,10);
		\draw (0,0)--(19,0);
		\draw[gray] (1,0)--(1,10);
		\draw[gray] (2,0)--(2,10);
		\draw[gray] (3,0)--(3,10);
		\draw[gray] (4,0)--(4,10);
		\draw[gray] (5,0)--(5,10);
		\draw[gray] (6,0)--(6,10);
		\draw[gray] (7,0)--(7,10);
		\draw[gray] (8,0)--(8,10);
		\draw[gray] (9,0)--(9,10);
		\draw[gray] (10,0)--(10,10);
		\draw[gray] (11,0)--(11,10);
		\draw[gray] (12,0)--(12,10);
		\draw[gray] (13,0)--(13,10);
		\draw[gray] (14,0)--(14,10);
		\draw[gray] (15,0)--(15,10);
		\draw[gray] (16,0)--(16,10);
		\draw[gray] (17,0)--(17,10);
		\draw[gray] (18,0)--(18,10);
		\draw[gray] (0,1)--(19,1);
		\draw[gray] (0,2)--(19,2);
		\draw[gray] (0,3)--(19,3);
		\draw[gray] (0,4)--(19,4);
		\draw[gray] (0,5)--(19,5);
		\draw[gray] (0,6)--(19,6);
		\draw[gray] (0,7)--(19,7);
		\draw[gray] (0,8)--(19,8);
		\draw[gray] (0,9)--(19,9);
		\draw(18,9) node[above right] {1};
		\draw(17,9) node[above right] {1};
		\draw(16,8) node[above right] {2};
		\draw(15,8) node[above right] {2};
		\draw(14,8) node[above right] {1};
		\draw(14,7) node[above right] {2};
		\draw(13,7) node[above right] {3};
		\draw(12,7) node[above right] {1};
		\draw(12,6) node[above right] {3};
		\draw(11,6) node[above right] {4};
		\draw(10,6) node[above right] {2};
		\draw(10,5) node[above right] {3};
		\draw(9,5) node[above right] {5};
		\draw(8,5) node[above right] {2};
		\draw(8,4) node[above right] {3};
		\draw(7,4) node[above right] {4};
		\draw(6,4) node[above right] {2};
		\draw(6,3) node[above right] {2};
		\draw(5,3) node[above right] {3};
		\draw(4,3) node[above right] {1};
		\draw(4,2) node[above right] {2};
		\draw(3,2) node[above right] {2};
		\draw(2,2) node[above right] {1};
		\draw(2,1) node[above right] {1};
		\draw(1,1) node[above right] {1};
		\draw(0,0) node[above right] {1};
	\end{tikzpicture}
	% \vspace{-.2in}
	\end{center}
	\caption{$\H\ast\ast(\Gr_2\R^{11,5})$}
\end{figure}

\newpage
\subsection{$q=6$}

\begin{figure}[ht!]
	\begin{center}
	\begin{tikzpicture}[scale=.4]
		\draw (0,0)--(0,11);
		\draw (0,0)--(21,0);
		\draw[gray] (1,0)--(1,11);
		\draw[gray] (2,0)--(2,11);
		\draw[gray] (3,0)--(3,11);
		\draw[gray] (4,0)--(4,11);
		\draw[gray] (5,0)--(5,11);
		\draw[gray] (6,0)--(6,11);
		\draw[gray] (7,0)--(7,11);
		\draw[gray] (8,0)--(8,11);
		\draw[gray] (9,0)--(9,11);
		\draw[gray] (10,0)--(10,11);
		\draw[gray] (11,0)--(11,11);
		\draw[gray] (12,0)--(12,11);
		\draw[gray] (13,0)--(13,11);
		\draw[gray] (14,0)--(14,11);
		\draw[gray] (15,0)--(15,11);
		\draw[gray] (16,0)--(16,11);
		\draw[gray] (17,0)--(17,11);
		\draw[gray] (18,0)--(18,11);
		\draw[gray] (19,0)--(19,11);
		\draw[gray] (20,0)--(20,11);
		\draw[gray] (0,1)--(21,1);
		\draw[gray] (0,2)--(21,2);
		\draw[gray] (0,3)--(21,3);
		\draw[gray] (0,4)--(21,4);
		\draw[gray] (0,5)--(21,5);
		\draw[gray] (0,6)--(21,6);
		\draw[gray] (0,7)--(21,7);
		\draw[gray] (0,8)--(21,8);
		\draw[gray] (0,9)--(21,9);
		\draw[gray] (0,10)--(21,10);
		\draw(20,10) node[above right] {1};
		\draw(19,10) node[above right] {1};
		\draw(18,10) node[above right] {1};
		\draw(18,9) node[above right] {1};
		\draw(17,9) node[above right] {2};
		\draw(16,9) node[above right] {1};
		\draw(16,8) node[above right] {2};
		\draw(15,8) node[above right] {3};
		\draw(14,8) node[above right] {2};
		\draw(14,7) node[above right] {2};
		\draw(13,7) node[above right] {4};
		\draw(12,7) node[above right] {2};
		\draw(12,6) node[above right] {3};
		\draw(11,6) node[above right] {5};
		\draw(10,6) node[above right] {3};
		\draw(10,5) node[above right] {3};
		\draw(9,5) node[above right] {5};
		\draw(8,5) node[above right] {2};
		\draw(8,4) node[above right] {3};
		\draw(7,4) node[above right] {4};
		\draw(6,4) node[above right] {2};
		\draw(6,3) node[above right] {2};
		\draw(5,3) node[above right] {3};
		\draw(4,3) node[above right] {1};
		\draw(4,2) node[above right] {2};
		\draw(3,2) node[above right] {2};
		\draw(2,2) node[above right] {1};
		\draw(2,1) node[above right] {1};
		\draw(1,1) node[above right] {1};
		\draw(0,0) node[above right] {1};
	\end{tikzpicture}
% \begin{figure}[ht!]
% 	\vspace{-.2in}
	\end{center}
	\caption{$\H\ast\ast(\Gr_2\R^{12,6})$}
\end{figure}

\begin{figure}[ht!]
	\begin{center}
	\begin{tikzpicture}[scale=.4]
		\draw (0,0)--(0,12);
		\draw (0,0)--(23,0);
		\draw[gray] (1,0)--(1,12);
		\draw[gray] (2,0)--(2,12);
		\draw[gray] (3,0)--(3,12);
		\draw[gray] (4,0)--(4,12);
		\draw[gray] (5,0)--(5,12);
		\draw[gray] (6,0)--(6,12);
		\draw[gray] (7,0)--(7,12);
		\draw[gray] (8,0)--(8,12);
		\draw[gray] (9,0)--(9,12);
		\draw[gray] (10,0)--(10,12);
		\draw[gray] (11,0)--(11,12);
		\draw[gray] (12,0)--(12,12);
		\draw[gray] (13,0)--(13,12);
		\draw[gray] (14,0)--(14,12);
		\draw[gray] (15,0)--(15,12);
		\draw[gray] (16,0)--(16,12);
		\draw[gray] (17,0)--(17,12);
		\draw[gray] (18,0)--(18,12);
		\draw[gray] (19,0)--(19,12);
		\draw[gray] (20,0)--(20,12);
		\draw[gray] (21,0)--(21,12);
		\draw[gray] (22,0)--(22,12);
		\draw[gray] (0,1)--(23,1);
		\draw[gray] (0,2)--(23,2);
		\draw[gray] (0,3)--(23,3);
		\draw[gray] (0,4)--(23,4);
		\draw[gray] (0,5)--(23,5);
		\draw[gray] (0,6)--(23,6);
		\draw[gray] (0,7)--(23,7);
		\draw[gray] (0,8)--(23,8);
		\draw[gray] (0,9)--(23,9);
		\draw[gray] (0,10)--(23,10);
		\draw[gray] (0,11)--(23,11);
		\draw(22,11) node[above right] {1};
		\draw(21,11) node[above right] {1};
		\draw(20,10) node[above right] {2};
		\draw(19,10) node[above right] {2};
		\draw(18,10) node[above right] {1};
		\draw(18,9) node[above right] {2};
		\draw(17,9) node[above right] {3};
		\draw(16,9) node[above right] {1};
		\draw(16,8) node[above right] {3};
		\draw(15,8) node[above right] {4};
		\draw(14,8) node[above right] {2};
		\draw(14,7) node[above right] {3};
		\draw(13,7) node[above right] {5};
		\draw(12,7) node[above right] {2};
		\draw(12,6) node[above right] {4};
		\draw(11,6) node[above right] {6};
		\draw(10,6) node[above right] {3};
		\draw(10,5) node[above right] {3};
		\draw(9,5) node[above right] {5};
		\draw(8,5) node[above right] {2};
		\draw(8,4) node[above right] {3};
		\draw(7,4) node[above right] {4};
		\draw(6,4) node[above right] {2};
		\draw(6,3) node[above right] {2};
		\draw(5,3) node[above right] {3};
		\draw(4,3) node[above right] {1};
		\draw(4,2) node[above right] {2};
		\draw(3,2) node[above right] {2};
		\draw(2,2) node[above right] {1};
		\draw(2,1) node[above right] {1};
		\draw(1,1) node[above right] {1};
		\draw(0,0) node[above right] {1};
	\end{tikzpicture}
	% \vspace{-.2in}
	\end{center}
	\caption{$\H\ast\ast(\Gr_2\R^{13,6})$}
\end{figure}

\newpage

We stop here, as generating all possible outcomes of even the lowest-tension of the 1716 distinct $E_1$ pages of the $\Gr_2\R^{14,7}$ eludes our computing power.

%%%%%%%%%%%%%%%%%%%%%%%%%%%%%%%%
%%%% FINISHING 363
%%%%%%%%%%%%%%%%%%%%%%%%%%%%%%%%

\section{Pushing a bit further}
\label{finishing363}

The algorithm outlined in the Section \ref{section:alg} can only narrow down the value of $\P(\H\ast\ast(\Gr_3(\R^{6,3})))$ to one of the following six possibilities:

\begin{enumerate}[(i)]
	\item {\tiny $x^9y^5 + x^8y^4 + 2x^7y^4 + x^6y^4 + x^5y^5 + 2x^6y^3 + 2x^5y^3 + 3x^4y^2 + 3x^3y^2 + x^2y^2 + x^2y + xy + 1$},
	\item  {\tiny $x^9y^5 + x^8y^4 + 2x^7y^4 + x^6y^4 + 2x^6y^3 + 3x^5y^3 + x^4y^4 + 2x^4y^2 + 3x^3y^2 + x^2y^2 + x^2y + xy + 1,$}
	\item  {\tiny $x^9y^5 + x^8y^4 + 2x^7y^4 + x^6y^4 + 2x^6y^3 + x^5y^4 + 2x^5y^3 + 3x^4y^2 + x^3y^3 + 2x^3y^2 + x^2y^2 + x^2y + xy + 1,$}
	\item  {\tiny $x^9y^5 + x^8y^4 + 2x^7y^4 + x^6y^4 + 2x^6y^3 + 3x^5y^3 + x^4y^3 + 2x^4y^2 + x^3y^3 + 2x^3y^2 + x^2y^2 + x^2y + xy + 1,$}
	\item  {\tiny $x^9y^5 + x^8y^4 + 2x^7y^4 + x^6y^4 + 2x^6y^3 + x^5y^4 + 2x^5y^3 + 3x^4y^2 + 3x^3y^2 + 2x^2y^2 + xy + 1$}, or
	\item  {\tiny $x^9y^5 + x^8y^4 + 2x^7y^4 + x^6y^4 + 2x^6y^3 + 3x^5y^3 + x^4y^3 + 2x^4y^2 + 3x^3y^2 + 2x^2y^2 + xy + 1$.}
\end{enumerate}

However, we can further narrow the possibilities by including as a subspace the smaller Grassmannian $\Gr_3\R^{5,2}\iso \Gr_2\R^{5,2}$ whose cohomology is known, and then quotienting:

\[\Gr_3(\R^{--+++})\inj \Gr_3(\R^{--+++-})\to Q:=\frac{\Gr_3(\R^{--+++})}{\Gr_3(\R^{--+++-})}.\]

Notice that while we do not know the cohomology of $Q$, we can build the first page $E_1(Q)$ of its spectral sequence by simply excluding those summands of $E_1(\Gr_3\R^{--+++-})$ which already appear in $E_1\Gr_3(\R^{--+++})$. Since $\H\ast\ast (Q)$ is some relaxation of $E_1(Q)$, we know $\H\ast\ast(\Gr_3\R^{6,3})$ is some relaxation of $ \H\ast\ast(\Gr_3\R^{5,2})\oplus E_1(Q)$. 

We compute that $\P(\H\ast\ast(\Gr_3\R^{5,2})\oplus E_1Q)=x^9y^5 + x^8y^4 + x^6y^6 + 2x^7y^4 + 2x^6y^3 + 3x^5y^3 + 3x^4y^2 + 3x^3y^2 + x^2y^2 + x^2y + xy + 1$ and learn that
only options (iv) and (vi) above qualify. See Figure \ref{fig:twoposs}.

\begin{figure}[ht]
	\begin{center}
	\begin{tikzpicture}[scale=.4]
		\draw (0,0)--(0,6);
		\draw (0,0)--(10,0);
		\draw[gray] (1,0)--(1,6);
		\draw[gray] (2,0)--(2,6);
		\draw[gray] (3,0)--(3,6);
		\draw[gray] (4,0)--(4,6);
		\draw[gray] (5,0)--(5,6);
		\draw[gray] (6,0)--(6,6);
		\draw[gray] (7,0)--(7,6);
		\draw[gray] (8,0)--(8,6);
		\draw[gray] (9,0)--(9,6);
		\draw[gray] (0,1)--(10,1);
		\draw[gray] (0,2)--(10,2);
		\draw[gray] (0,3)--(10,3);
		\draw[gray] (0,4)--(10,4);
		\draw[gray] (0,5)--(10,5);
		\draw(9,5) node[above right] {1};
		\draw(8,4) node[above right] {1};
		\draw(7,4) node[above right] {2};
		\draw(6,4) node[above right] {1};
		\draw(6,3) node[above right] {2};
		\draw(5,3) node[above right] {3};
		\draw(4,3) node[above right] {1};
		\draw(4,2) node[above right] {2};
		\draw(3,3) node[above right] {1};
		\draw(3,2) node[above right] {2};
		\draw(2,2) node[above right] {1};
		\draw(2,1) node[above right] {1};
		\draw(1,1) node[above right] {1};
		\draw(0,0) node[above right] {1};
		\draw (5,-1) node {(iv)};
		\color{violet}
		\draw (2,1)--(3,1)--(3,2)--(2,2)--(2,1);
	\end{tikzpicture}
	\qquad
	\begin{tikzpicture}[scale=.4]
		\draw (0,0)--(0,6);
		\draw (0,0)--(10,0);
		\draw[gray] (1,0)--(1,6);
		\draw[gray] (2,0)--(2,6);
		\draw[gray] (3,0)--(3,6);
		\draw[gray] (4,0)--(4,6);
		\draw[gray] (5,0)--(5,6);
		\draw[gray] (6,0)--(6,6);
		\draw[gray] (7,0)--(7,6);
		\draw[gray] (8,0)--(8,6);
		\draw[gray] (9,0)--(9,6);
		\draw[gray] (0,1)--(10,1);
		\draw[gray] (0,2)--(10,2);
		\draw[gray] (0,3)--(10,3);
		\draw[gray] (0,4)--(10,4);
		\draw[gray] (0,5)--(10,5);
		\draw(9,5) node[above right] {1};
		\draw(8,4) node[above right] {1};
		\draw(7,4) node[above right] {2};
		\draw(6,4) node[above right] {1};
		\draw(6,3) node[above right] {2};
		\draw(5,3) node[above right] {3};
		\draw(4,3) node[above right] {1};
		\draw(4,2) node[above right] {2};
		\draw(3,2) node[above right] {3};
		\draw(2,2) node[above right] {2};
		\draw(1,1) node[above right] {1};
		\draw(0,0) node[above right] {1};
		\draw (5,-1) node {(vi)};
		\color{violet}
		\draw (2,1)--(3,1)--(3,2)--(2,2)--(2,1);
	\end{tikzpicture}
	\end{center}
	\caption{The two remaining possibilities for $\H\ast\ast(\Gr_3\R^{6,3})$. We will focus on bidegree {\color{violet} $(2,1)$}.}
	\label{fig:twoposs}
\end{figure}

Finally, note that (vi)-(iv)$=x^2yK_{1,1}$, and so we need only determine whether or not this one shift occurs. For this, we again make use of a known cohomology, this time of a (much) larger space, $\Gr_3((\R^{+-})^{\oplus\infty})$, whose cohomology was calculated by Dugger \cite{Dan}. This cohomology begins as shown in Figure \ref{fig:Gr3U}.

\begin{figure}[ht] 
	\begin{center}
	\begin{tikzpicture}[scale=.5]
		\draw (0,0)--(0,6);
		\draw (0,0)--(9,0);
		\draw[gray] (1,0)--(1,6);
		\draw[gray] (2,0)--(2,6);
		\draw[gray] (3,0)--(3,6);
		\draw[gray] (4,0)--(4,6);
		\draw[gray] (5,0)--(5,6);
		\draw[gray] (6,0)--(6,6);
		\draw[gray] (7,0)--(7,6);
		\draw[gray] (8,0)--(8,6);
		\draw[gray] (0,1)--(9,1);
		\draw[gray] (0,2)--(9,2);
		\draw[gray] (0,3)--(9,3);
		\draw[gray] (0,4)--(9,4);
		\draw[gray] (0,5)--(9,5);
		\draw(8,5) node[above right] {6};
		\draw(8,4) node[above right] {4};
		\draw(7,5) node[above right] {2};
		\draw(7,4) node[above right] {6};
		\draw(6,4) node[above right] {4};
		\draw(6,3) node[above right] {3};
		\draw(5,4) node[above right] {1};
		\draw(5,3) node[above right] {4};
		\draw(4,3) node[above right] {2};
		\draw(4,2) node[above right] {2};
		\draw(3,3) node[above right] {1};
		\draw(3,2) node[above right] {2};
		\draw(2,2) node[above right] {1};
		\draw(2,1) node[above right] {1};
		\draw(1,1) node[above right] {1};
		\draw(0,0) node[above right] {1};
		\draw(9,5) node[above right] {$\dots$};
		\draw(9,6) node[above right] {$\reflectbox{$\ddots$}$};
		\color{violet}
		\draw (2,1)--(3,1)--(3,2)--(2,2)--(2,1);
	\end{tikzpicture}
	\end{center}
	\caption{Some of $\H\ast\ast \Gr_3(\R^{+-})^{\oplus\infty}$}
	\label{fig:Gr3U}
\end{figure}

Consider the cofiber sequence

\[\Gr_3\R^{6,3}\xhookrightarrow{i} \Gr_3((\R^{+-})^{\oplus\infty})\to R:=\frac{\Gr_3(\R^{+-})^{\oplus\infty}}{\Gr_3\R^{6,3}}.\]

Note that in singular cohomology $H_{\text{sing}}^2(R)=H_{\text{sing}}^3(R)=0$ and so in the long exact sequence of this cofiber sequence we get $$i^*_{\text{sing}}:H_{\text{sing}}^2\Gr_3(\R^6)\iso H_{\text{sing}}^2(\Gr \R^\infty)\iso \Z_2\oplus \Z_2.$$ Now suppose for a contradiction that $\H21(\Gr_3\R^{6,3})=0$ as in case (vi). Then necessarily $(i^*)^{2,1}=0$ and so $(i^*_{\text{sing}})^2=\psi((i^*)^{2,1})\oplus \psi((i^*)^{2,2})$ cannot be full rank, and we have our contradiction. And so only option (iv) remains. See Figure \ref{fig:363}.

\begin{figure}[ht]
	\begin{center}
	\begin{tikzpicture}[scale=.5]
		\draw (0,0)--(0,6);
		\draw (0,0)--(10,0);
		\draw[gray] (1,0)--(1,6);
		\draw[gray] (2,0)--(2,6);
		\draw[gray] (3,0)--(3,6);
		\draw[gray] (4,0)--(4,6);
		\draw[gray] (5,0)--(5,6);
		\draw[gray] (6,0)--(6,6);
		\draw[gray] (7,0)--(7,6);
		\draw[gray] (8,0)--(8,6);
		\draw[gray] (9,0)--(9,6);
		\draw[gray] (0,1)--(10,1);
		\draw[gray] (0,2)--(10,2);
		\draw[gray] (0,3)--(10,3);
		\draw[gray] (0,4)--(10,4);
		\draw[gray] (0,5)--(10,5);
		\draw(9,5) node[above right] { 1};
		\draw(8,4) node[above right] { 1};
		\draw(7,4) node[above right] { 2};
		\draw(6,4) node[above right] { 1};
		\draw(6,3) node[above right] { 2};
		\draw(5,3) node[above right] { 3};
		\draw(4,3) node[above right] { 1};
		\draw(4,2) node[above right] { 2};
		\draw(3,3) node[above right] { 1};
		\draw(3,2) node[above right] { 2};
		\draw(2,2) node[above right] { 1};
		\draw(2,1) node[above right] { 1};
		\draw(1,1) node[above right] { 1};
		\draw(0,0) node[above right] { 1};
	\end{tikzpicture}
	\end{center}
	\caption{$\H\ast\ast(\Gr_3\R^{6,3})$}
	\label{fig:363}
\end{figure}

\begin{comment}
	Note that Figure \ref{fig:363} is the only result in this paper that is not fully ``relaxed,'' that is, for which a Kronholm shift might have occurred but didn't. It may be that our algorithm from Section \ref{section:alg} yields conclusive results \emph{only} when every possible shift does in fact occur, and also when some $E_1$ page exists which provides evidence of its occurrence. Further calculations like the one in this section might cast some light.
\end{comment}

\printbibliography

\end{document}